\documentclass[12pt]{amsart}
\usepackage{epsfig}
\usepackage{mathtools}
\usepackage[subnum]{cases}
\usepackage{enumerate}
\usepackage{bbm,bm}
\usepackage{amsmath,amssymb,mathrsfs,amsthm}
\usepackage{color}
\usepackage{verbatim}
\usepackage{pgf,tikz,pgfplots,tikz-3dplot-circleofsphere,tkz-euclide}
\usetikzlibrary{intersections}
\pgfplotsset{compat=1.14}
\usetikzlibrary{arrows, shapes.geometric}
\usepackage{graphicx}
\usepackage{ctable}
\graphicspath{ {./images/} }
\usepackage{hyperref}

\usepackage[hang]{footmisc} 

\newtheorem{theorem}{Theorem}[section]

\newtheorem{lemma}[theorem]{Lemma}

\newtheorem{remark}[theorem]{Remark}

\newtheorem*{property*}{Property}

\newcommand{\R}{\mathbb{R}} 
\newcommand{\Sp}{\mathbb{S}^{2}}

\DeclareMathOperator*{\argmax}{argmax}

\DeclareMathOperator{\vol}{vol}
\DeclareMathOperator{\area}{area}

\DeclareMathOperator{\conv}{conv}

\makeatletter
\def\hlinewd#1{%
\noalign{\ifnum0=`}\fi\hrule \@height #1 \futurelet
\reserved@a\@xhline}
\makeatother

\definecolor{electricviolet}{rgb}{0.56, 0.0, 1.0}
\begin{document}
\setcounter{footnote}{0}

\title[The maximum volume polytope with nine vertices inscribed in the sphere]{The maximum volume polytope  with nine vertices inscribed in the sphere}

    \author[S. Hoehner and J. Ledford]{Steven Hoehner and Jeff Ledford}
    \date{\today}


\begin{abstract}
A classical problem in convex and discrete geometry asks for the convex polyhedron of greatest volume whose vertices are chosen from the unit sphere $\mathbb{S}^2$. For a prescribed number $N$ of vertices, the problem is known only in a small number of cases, namely $4\leq N\leq 8$ and $N=12$. In the case $N=9$, Wimmer showed in his Ph.D. thesis that the volume maximizer must be combinatorially equivalent to a triaugmented triangular prism. In this note, we complete the proof by determining the volume maximizer among all inscribed polytopes which are combinatorially equivalent to a triaugmented triangular prism. Consequently, every convex polyhedron with at most nine vertices on $\mathbb{S}^2$ has volume at most $3\sqrt{2\sqrt{3}-3}$,
with equality, up to rotation, precisely for the ``symmetric'' triaugmented triangular prism whose vertex coordinates we compute exactly.
\end{abstract}
\maketitle

\section{Introduction and main results}\label{sec:introduction}

A classical extremal problem in convex and discrete geometry asks the following: among all convex polyhedra with a prescribed number of vertices on the unit sphere, which one has greatest volume? More precisely, for an integer $N\geq 4$, let
\[
\mathcal{P}_N
:=\left\{P\subset\mathbb{R}^3:
P\text{ is a convex polytope with at most }N
\text{ vertices, all contained in }\mathbb{S}^2\right\},
\]
where
\[
\mathbb{S}^2
=\{(x_1,x_2,x_3)\in\mathbb{R}^3: x_1^2+x_2^2+x_3^2=1\}
\]
is the unit sphere in $\R^3$. The problem is to determine
\[
\max_{P\in\mathcal{P}_N}\vol(P),
\]
and to characterize all polytopes for which equality occurs.

Despite the elementary formulation of this problem, exact solutions are known for only a small number of values of $N$. The cases $N=4,5,6$ and $N=12$ are classical (see, e.g., \cite{Toth}), while Berman and Hanes \cite{BermanHanes1970} determined the maximizers for $N=7$ and $N=8$. Their work introduced structural and local optimality methods that remain fundamental in the study of the problem; see, e.g., \cite{HorvathLangi,Wimmer1998}. Moreover, their result in the case $N=8$ confirms an earlier numerical result of Grace \cite{Grace-1963}, who discovered the same polytope as a local maximizer using a computer, and conjectured it could be the global maximizer. In fact, it is believed that this polytope could be the first shape discovered by a computer, as described in the recent exposition of Parker \cite{Parker-2024-BiggestShape}. This polytope is thus sometimes called \emph{Grace's polyhedron}; see Figure \ref{fig:global-volume-maximizers} below.

In the case $N=9$, Wimmer  proved in his Ph.D. thesis \cite{Wimmer1998} that the volume maximizer in $\mathcal{P}_9$ must be combinatorially equivalent to a triaugmented triangular prism. Wimmer \cite{Wimmer1998} also determined the volume maximizer in the subclass of symmetric triaugmented triangular prisms. In this note, we complete the proof of the $N=9$ case by determining the volume maximizer among all inscribed polytopes in $\mathcal{P}_9$ which are combinatorially equivalent to a triaugmented triangular prism, with no symmetry assumptions. Combined with Wimmer's reduction, this yields the following result.

\begin{theorem}\label{mainThm}
Let $P\in\mathcal{P}_9$. Then
\[
\vol(P)
\leq 3\sqrt{2\sqrt{3}-3}
\approx 2.04375.
\]
Equality holds if and only if, up to rotation, $P$ is the triaugmented triangular prism with vertices
\begin{align*}
a_1&=(1,0,0),&
a_2&=\left(-\tfrac{1}{2},\tfrac{\sqrt{3}}{2},0\right),&
a_3&=\left(-\tfrac{1}{2},-\tfrac{\sqrt{3}}{2},0\right),\\
x_1&=\left(\tfrac{1}{2}r_*,\tfrac{\sqrt{3}}{2}r_*,h_*\right),&
x_2&=(-r_*,0,h_*),&
x_3&=\left(\tfrac{1}{2}r_*,-\tfrac{\sqrt{3}}{2}r_*,h_*\right),\\
y_1&=\left(\tfrac{1}{2}r_*,\tfrac{\sqrt{3}}{2}r_*,-h_*\right),&
y_2&=(-r_*,0,-h_*),&
y_3&=\left(\tfrac{1}{2}r_*,-\tfrac{\sqrt{3}}{2}r_*,-h_*\right),
\end{align*}
where
\[
h_*:=\sqrt{2\sqrt3-3}
\qquad\text{and}\qquad
r_*:=\sqrt{1-h_*^2}
=\sqrt{4-2\sqrt3}
 =\sqrt{3}-1 .\]
\end{theorem}

The maximizing polytope has the combinatorial type of the triaugmented triangular prism, also known as Johnson solid $J_{51}$. We emphasize, however, that the initial problem is not restricted a priori to this or any other combinatorial type: the maximization in Theorem \ref{mainThm} is taken over all convex polytopes with at most nine vertices on the sphere.

For the triaugmented triangular prism class, we explicitly determine the volume maximizer without using any symmetry assumptions. Writing the six vertices of the underlying triangular prism as $x_1$, $x_2$, $x_3$, $y_1$, $y_2$, $y_3$, and the three augmenting vertices as $a_1,a_2,a_3$, we use an oriented determinant decomposition and the Cauchy--Schwarz inequality to reduce the volume estimate to a six-vector optimization problem. Then, after changing to suitable orthogonal coordinates, we reduce this problem to a two-variable inequality, whose unique equality configuration forces the two triangular bases to be parallel congruent equilateral triangles, symmetrically situated about the origin, while the three augmenting vertices form an equilateral triangle in the intermediate plane. Finally, maximizing over the remaining variable, we derive that
\[
h_*=\sqrt{2\sqrt{3}-3}.
\]
Furthermore, we show that the equality conditions recover, precisely, the polytope stated in Theorem \ref{mainThm}.



    \subsection{Related  results}

In the following table, we list the known complete solutions of the volume maximization problem for $N\geq 4$ points on the unit sphere $\mathbb{S}^2$.  The second column lists the global volume maximizer with $N$ vertices, and the third column gives the corresponding volume of the  maximizer. 

\vspace{2mm}

\begin{center}
    \begin{tabular}{cccc}
    $N$   & $\argmax_{P\in\mathcal{P}_N}\vol(P)$ & $\max_{P\in\mathcal{P}_N}\vol(P)$ & Citation \\ \specialrule{.2em}{.1em}{.1em} 
    4  & regular tetrahedron & $8\sqrt{3}/27\approx 0.513$  &  e.g., \cite[p. 263]{Toth}\\ \hline
    5  & triangular bipyramid &  $\sqrt{3}/2\approx 0.866$  & e.g., \cite{BermanHanes1970}\\ \hline
    6  & regular octahedron   & $4/3\approx 1.333$   & e.g., \cite[p. 263]{Toth}\\ \hline
    7  & pentagonal bipyramid   & $\frac{5}{12}\sqrt{10+2\sqrt{5}}\approx 1.585$ & \cite{BermanHanes1970}   \\ \hline
    8  &  Grace's polyhedron & $\sqrt{\frac{475+29\sqrt{145}}{250}}\approx 1.816$ & \cite{BermanHanes1970} \\\hline
    9  & triaugmented triangular prism & $3\sqrt{2\sqrt{3}-3}\approx 2.044$ & \cite{Wimmer1998}, Theorem \ref{mainThm} \\ \hline
    10 & -- & -- & --\\ \hline
    11  &  -- & -- & --\\ \hline
    12 & regular icosahedron  & $\frac{2}{3}\sqrt{10+2\sqrt{5}}\approx 2.536$  & e.g., \cite[p. 263]{Toth}\\ \hline
    $\geq 13$ & -- & -- & --
    \end{tabular}
    \end{center}
    
\vspace{5mm}
    
    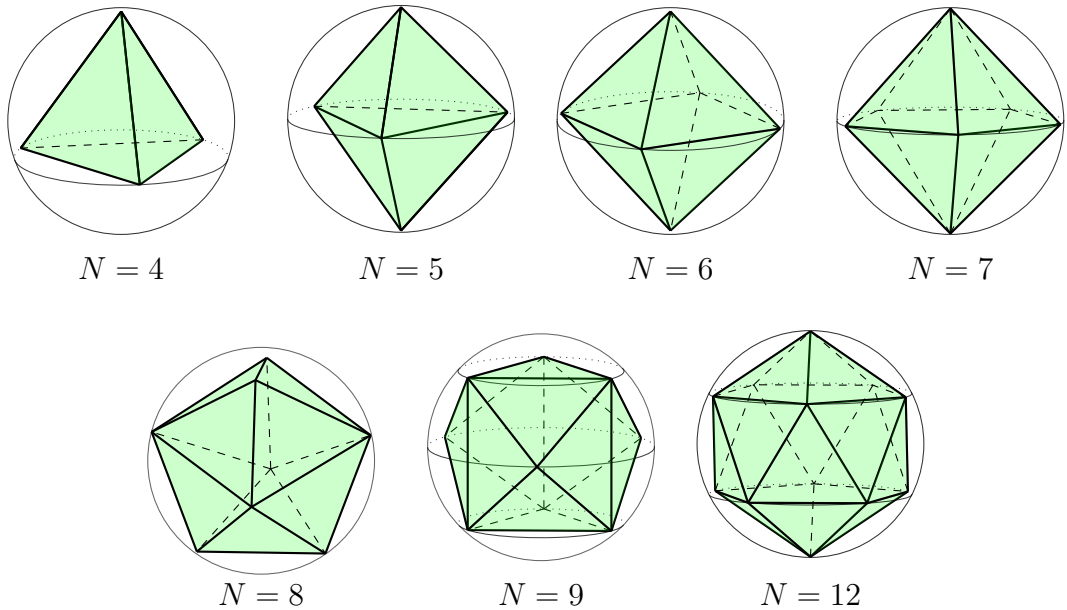
\begin{figure}
    \begin{center}
    \def\r{1}
    \tdplotsetmaincoords{75}{80}
  \begin{tikzpicture}[scale=1.5,line join=bevel, tdplot_main_coords]
    \coordinate (O) at (0,0,0);

\coordinate (A) at ({sqrt(8/9)},0,{-1/3});
\coordinate (B) at ({-sqrt(2/9)},{sqrt(2/3)},{-1/3});
\coordinate (C) at ({-sqrt(2/9)},{-sqrt(2/3)},{-1/3});
\coordinate (D) at (0,0,1);

\begin{scope}[thick]
    \draw (A) -- (D)--(B);
    \draw (A) -- (D)--(C);
    \draw (A)--(B)--(D);
    \draw (A)--(C)--(D);
\end{scope}

\draw[thick,fill=green, opacity=0.2] (A) -- (D)--(B);
\draw[thick,fill=green, opacity=0.2] (A) -- (C) -- (D);

\begin{scope}[dashed] 
    \draw (B) -- (C);
\end{scope}

\begin{scope}[opacity=0.8]
\draw[tdplot_screen_coords] (0,0,0) circle (\r);
\tdplotCsDrawLatCircle{\r}{-19.4712206345}
\end{scope} 

 \coordinate (L1) at (0,0,-1);
    \node[yshift=-5mm] at (L1) {$N=4$};
\end{tikzpicture}
\tdplotsetmaincoords{80}{100}
\def\r{1}
  \begin{tikzpicture}[scale=1.5,line join=bevel, tdplot_main_coords]
    \coordinate (O) at (0,0,0);

\coordinate (A) at (1,0,0);
\coordinate (B) at ({-1/2},{sqrt(3)/2},0);
\coordinate (C) at ({-1/2},{-sqrt(3)/2},0);
\coordinate (D) at (0,0,1);
\coordinate (E) at (0,0,{-1});

\begin{scope}[thick]
    \draw (A) -- (D)--(B);
    \draw (A) -- (D) -- (C);
    \draw (A) -- (B)--(E);
    \draw (A)--(C)--(E);
    \draw (A)--(E)--(B);
\end{scope}

\draw[thick,fill=green, opacity=0.2] (A) -- (D)--(B);
\draw[thick,fill=green, opacity=0.2] (A) -- (D) -- (C);
\draw[thick,fill=green,opacity=0.2](A) -- (C) -- (E);  
\draw[thick,fill=green,opacity=0.2] (A)--(E)--(B);  

\begin{scope}[dashed] 
    \draw (C) -- (B);
\end{scope}

\begin{scope}[opacity=0.8]
\draw[tdplot_screen_coords] (0,0,0) circle (\r);
\tdplotCsDrawLatCircle{\r}{0}
\end{scope} 

 \coordinate (L1) at (0,0,-1);
    \node[yshift=-5mm] at (L1) {$N=5$};
  \end{tikzpicture}
  \tdplotsetmaincoords{75}{105}
  \begin{tikzpicture}[scale=1.5,line join=bevel, tdplot_main_coords]
    \coordinate (O) at (0,0,0);

\coordinate (A) at (1,0,0);
\coordinate (B) at (0,1,0);
\coordinate (C) at ({-1},0,0);
\coordinate (D) at (0,{-1},0);
\coordinate (E) at (0,0,1);
\coordinate (F) at (0,0,{-1});
);

\begin{scope}[thick]
    \draw (B)--(E);
    \draw (B)--(F);
    \draw (D)--(E);
    \draw (D)--(F);
    \draw (A)--(E);
    \draw (A)--(F);
    \draw (B)--(A)--(D);
\end{scope}

\draw[thick,fill=green, opacity=0.2] (B) -- (A)--(E);
\draw[thick,fill=green, opacity=0.2] (A) -- (D) -- (E);
\draw[thick,fill=green,opacity=0.2](B) -- (A) -- (F);  
\draw[thick,fill=green,opacity=0.2] (A)--(D)--(F);  

\begin{scope}[dashed] 
    \draw (C) -- (E);
    \draw (C)--(F);
    \draw (D)--(C);
    \draw (B)--(C);
\end{scope}

\begin{scope}[opacity=0.8]
\draw[tdplot_screen_coords] (0,0,0) circle (\r);
\tdplotCsDrawLatCircle{\r}{0}
\end{scope} 

 \coordinate (L1) at (0,0,-1);
    \node[yshift=-5mm] at (L1) {$N=6$};
  \end{tikzpicture}
  \tdplotsetmaincoords{83}{86}
\def\r{1}
\begin{tikzpicture}[scale=1.5,line join=bevel, tdplot_main_coords]
    \coordinate (O) at (0,0,0);

\coordinate (A) at (1,0,0);
\coordinate (B) at ({(-1+sqrt(5))/4},{sqrt((5+sqrt(5))/8)},0);
\coordinate (C) at ({(-1-sqrt(5))/4},{sqrt((5-sqrt(5))/8)},0);
\coordinate (D) at (0,0,1);
\coordinate (E) at (0,0,{-1});
\coordinate (F) at ({(-1-sqrt(5))/4},{-sqrt((5-sqrt(5))/8)},0);
\coordinate (G) at ({(-1+sqrt(5))/4},-{sqrt((5+sqrt(5))/8)},0);

\begin{scope}[thick]
    \draw (A) -- (D)--(B);
    \draw (A) -- (B)--(E);
    \draw (G)--(A)--(E);
    \draw (D)--(G);
    \draw (G)--(E);
\end{scope}

\draw[thick,fill=green, opacity=0.2] (A) -- (D)--(B);
\draw[thick,fill=green, opacity=0.2] (A) -- (D) -- (G);
\draw[thick,fill=green,opacity=0.2](A) -- (G) -- (E);  
\draw[thick,fill=green,opacity=0.2] (A)--(E)--(B);  

\begin{scope}[dashed] 
    \draw (C) -- (B);
    \draw (D)--(C);
    \draw (D)--(F);
    \draw (C)--(F);
    \draw (F)--(G);
    \draw (E)--(F);
    \draw (E)--(C);
\end{scope}

\begin{scope}[opacity=0.8]
\draw[tdplot_screen_coords] (0,0,0) circle (\r);
\tdplotCsDrawLatCircle{\r}{0}
\end{scope} 

\node[yshift=-5mm] at (L1) {$N=7$};
  \end{tikzpicture}

  \vspace{5mm}
  
  \tdplotsetmaincoords{80}{95}
\def\r{1}

\pgfmathsetmacro{\phiGrace}{acos(sqrt((15+sqrt(145))/40))}

\begin{tikzpicture}[scale=1.5,line join=bevel,tdplot_main_coords]

    \coordinate (O) at (0,0,0);

    \coordinate (A1) at
    ({sin(3*\phiGrace)},0,{cos(3*\phiGrace)});

    \coordinate (B1) at
    ({-sin(3*\phiGrace)},0,{cos(3*\phiGrace)});

    \coordinate (C1) at
    ({sin(\phiGrace)},0,{cos(\phiGrace)});

    \coordinate (D1) at
    ({-sin(\phiGrace)},0,{cos(\phiGrace)});

    \coordinate (E1) at
    (0,{sin(3*\phiGrace)},{-cos(3*\phiGrace)});

    \coordinate (F1) at
    (0,{-sin(3*\phiGrace)},{-cos(3*\phiGrace)});

    \coordinate (G1) at
    (0,{sin(\phiGrace)},{-cos(\phiGrace)});

    \coordinate (H1) at
    (0,{-sin(\phiGrace)},{-cos(\phiGrace)});

    \begin{scope}[opacity=0.6]
        \draw[tdplot_screen_coords] (0,0,0) circle (\r);
    \end{scope}

    \begin{scope}[fill=green,opacity=0.2]
        \fill (C1)--(F1)--(A1)--cycle;
        \fill (E1)--(C1)--(A1)--cycle;
        \fill (H1)--(F1)--(A1)--cycle;
        \fill (D1)--(C1)--(F1)--cycle;
        \fill (D1)--(E1)--(C1)--cycle;
        \fill (G1)--(E1)--(A1)--cycle;
        \fill (G1)--(H1)--(A1)--cycle;
    \end{scope}

    \begin{scope}[dashed]
        \draw (B1)--(D1);
        \draw (B1)--(E1);
        \draw (B1)--(F1);
        \draw (B1)--(G1);
        \draw (B1)--(H1);
    \end{scope}

    \begin{scope}[thick]
        \draw (A1)--(C1);

        \draw (A1)--(E1);
        \draw (A1)--(F1);
        \draw (A1)--(G1);
        \draw (A1)--(H1);

        \draw (C1)--(D1);
        \draw (C1)--(E1);
        \draw (C1)--(F1);

        \draw (D1)--(E1);
        \draw (D1)--(F1);

        \draw (E1)--(G1);
        \draw (F1)--(H1);

        \draw (G1)--(H1);
    \end{scope}


    \node[tdplot_screen_coords] at (0,-1.18) {$N=8$};

\end{tikzpicture}
  \tdplotsetmaincoords{80}{92}
\def\r{1}
  \begin{tikzpicture}[scale=1.5,line join=bevel, tdplot_main_coords]
    \coordinate (O) at (0,0,0);

\coordinate (A1) at ({sqrt(4-2*sqrt(3))*cos(60)},{sqrt(4-2*sqrt(3))*sin(60)},{sqrt(2*sqrt(3)-3)});
\coordinate (B1) at ({sqrt(4-2*sqrt(3))*cos(180)},{sqrt(4-2*sqrt(3))*sin(180)},{sqrt(2*sqrt(3)-3)});
\coordinate (C1) at ({sqrt(4-2*sqrt(3))*cos(300)},{sqrt(4-2*sqrt(3))*sin(300)},{sqrt(2*sqrt(3)-3)});

\begin{scope}[thick]
      \draw (B1)--(A1);
      \draw (A1)--(C1);
       \draw (C1) -- (B1);
\end{scope} 

\begin{scope}[opacity=0.6]
\draw[tdplot_screen_coords] (0,0,0) circle (\r);
\tdplotCsDrawLatCircle{\r}{42.94}
\tdplotCsDrawLatCircle{\r}{0}
\tdplotCsDrawLatCircle{\r}{-42.94}
\end{scope} 

\coordinate (D1) at ({sqrt(4-2*sqrt(3))*cos(60)},{sqrt(4-2*sqrt(3))*sin(60)},{-sqrt(2*sqrt(3)-3)});
\coordinate (E1) at ({sqrt(4-2*sqrt(3))*cos(180)},{sqrt(4-2*sqrt(3))*sin(180)},{-sqrt(2*sqrt(3)-3)});
\coordinate (F1) at ({sqrt(4-2*sqrt(3))*cos(300)},{sqrt(4-2*sqrt(3))*sin(300)},{-sqrt(2*sqrt(3)-3)});

\begin{scope}[thick]
       \draw (D1) -- (F1);
\end{scope}

\begin{scope}[dashed]
 \draw (E1)--(F1);
      \draw (D1)--(E1);
      \end{scope}

      \coordinate (G1) at (1,0,0);
\coordinate (H1) at ({cos(120)},{sin(120)},0);
\coordinate (I1) at ({cos(240)},{sin(240)},0);

\begin{scope}[thick]
 \draw (A1)--(D1);
 \draw (C1)--(F1);
\end{scope}

\begin{scope}[dashed]
     \draw (B1)--(E1);
      \end{scope}

\begin{scope}[thick]
\draw (A1)--(G1);
\draw (C1)--(G1);
\draw (D1)--(G1);
\draw (F1)--(G1);
\draw (D1)--(H1);
\draw (F1)--(I1);
\draw (A1)--(H1);
\draw (C1)--(I1);
\end{scope}

\begin{scope}[dashed]
\draw (E1)--(I1);
\draw (E1)--(H1);
\draw (B1)--(I1);
\draw (B1)--(H1);
\end{scope}

\draw[thick,fill=green,opacity=0.2] (A1)--(G1)--(C1);
\draw[thick,fill=green,opacity=0.2] 
(A1)--(G1)--(D1);
\draw[thick,fill=green,opacity=0.2] 
(A1)--(B1)--(C1);
\draw[thick,fill=green,opacity=0.2] 
(C1)--(G1)--(F1);
\draw[thick,fill=green,opacity=0.2] 
(D1)--(G1)--(F1);
\draw[thick,fill=green,opacity=0.2] 
(D1)--(A1)--(H1);
\draw[thick,fill=green,opacity=0.2] 
(C1)--(F1)--(I1);

    \node[yshift=-5mm] at (L1) {$N=9$};
  \end{tikzpicture}
  \tdplotsetmaincoords{84}{92}
  \begin{tikzpicture}[scale=1.5,line join=bevel, tdplot_main_coords]
   \coordinate (O) at (0,0,0);

\coordinate (A) at (0,0,1);
\coordinate (B) at (0.894427191,0,0.4472135955);
\coordinate (C) at (0.2763932023,0.8506508084,{1/sqrt(5)});
\coordinate (D) at ({-0.7236067977},0.5257311121,{1/sqrt(5});
\coordinate (E) at ({-0.7236067977},{-0.5257311121},{1/sqrt(5)});
\coordinate (F) at (0.2763932022,{-0.8506508084},{1/sqrt(5)});
\coordinate (G) at (0.7236067977,0.5257311121,{-1/sqrt(5)});
\coordinate (H) at ({-0.2763932023},0.8506508084,{-1/sqrt(5)});
\coordinate (I) at ({-0.894427191},0,{-1/sqrt(5)});
\coordinate (J) at ({-0.2763932022},{-0.8506508084},{-1/sqrt(5)});
\coordinate (K) at (0.7236067977,{-0.5257311121},{-1/sqrt(5)});
\coordinate (L) at (0,0,-1);

\begin{scope}[thick]
    \draw (A) -- (F) -- (B);
    \draw (A) -- (B) -- (C);
    \draw (L) -- (K) -- (B);
    \draw (L) -- (G) -- (B);
    \draw (J) -- (F) -- (K);
    \draw (H) -- (C) -- (G);
    \draw (J) -- (K) -- (G) -- (H);
    \draw (A) -- (C);
    \draw (H) -- (L) -- (J);
\end{scope}

\draw[thick,fill=green,opacity=0.2] (A)--(F)--(B);
\draw[thick,fill=green,opacity=0.2] (A)--(C)--(B);
\draw[thick,fill=green,opacity=0.2] (J)--(K)--(F);  
\draw[thick,fill=green,opacity=0.2] (F)--(K)--(B);
\draw[thick,fill=green,opacity=0.2] (B)--(K)--(G);
\draw[thick,fill=green,opacity=0.2] (B)--(G)--(C);
\draw[thick,fill=green,opacity=0.2] (C)--(H)--(G);
\draw[thick,fill=green,opacity=0.2] (L)--(K)--(G);
\draw[thick,fill=green,opacity=0.2] (L)--(K)--(J);
\draw[thick,fill=green,opacity=0.2] (L)--(G)--(H);

\begin{scope}[dashed] 
    \draw (F) -- (E) -- (D) -- (C);
    \draw (E) -- (A) -- (D);
    \draw (J) -- (I);
    \draw (E) -- (I) -- (D);
    \draw (L) -- (I);
    \draw (I) -- (H);
    \draw (H) -- (D);
    \draw (J) -- (E);
\end{scope}

\begin{scope}[opacity=0.8]
\draw[tdplot_screen_coords] (0,0,0) circle (\r);
\tdplotCsDrawLatCircle{\r}{26.5650512}
\tdplotCsDrawLatCircle{\r}{-26.5650512}
\end{scope} 

 \coordinate (L1) at (0,0,-1);
    \node[yshift=-5mm] at (L1) {$N=12$};
  \end{tikzpicture}
     \end{center}
     \caption{The global volume maximizers for $N\in\{4,5,6,7,8,9,12\}$.}
    \label{fig:global-volume-maximizers}
     \end{figure}

     \begin{remark}
Following Berman and Hanes \cite{BermanHanes1970}, a simplicial
polytope with $N$ vertices is called \emph{medial} if every vertex has
degree $\lfloor 6-12/N\rfloor$ or $\lceil 6-12/N\rceil$. For $N=9$, these degrees are $4$ and $5$. Since a simplicial
$9$-vertex polytope has $21$ edges, any medial polytope must have
exactly three vertices of degree $4$ and six vertices of degree $5$.
The triaugmented triangular prism appearing in Theorem
\ref{mainThm} has exactly this degree sequence, as its three augmenting
vertices have degree $4$, while the six vertices of the underlying
triangular prism have degree $5$. Thus, the maximizer in Theorem
\ref{mainThm} is medial. Goldberg \cite{Goldberg} conjectured that a maximum volume polyhedron inscribed in the sphere should be medial whenever a medial polyhedron exists.  In particular, the case $N=9$ in Theorem \ref{mainThm} provides another affirmative instance of the conjecture.
\end{remark}

     \begin{remark}
The corresponding volume maximization problem has also been studied in higher
dimensions. Horv\'ath and L\'angi \cite{HorvathLangi} extended the local optimality conditions of Berman and Hanes to polytopes inscribed in $\mathbb{S}^{d-1}$ and showed, in particular, that a volume maximizer is simplicial. They completely determined the maximizers with $d+2$ vertices in every dimension, showing that such a maximizer is the convex hull of two regular simplices of dimensions
$\lfloor d/2\rfloor$ and $\lceil d/2\rceil$, respectively, contained in mutually orthogonal complementary subspaces. They also solved the problem for $d+3$ vertices when $d$ is odd. In this case, the maximizer is the convex hull of three regular simplices contained in mutually orthogonal subspaces, whose dimensions differ by at most one. For even $d$, they obtained the analogous result among noncyclic polytopes, while the cyclic case remains open. In particular, when
$d=3$, their results recover the maximality of the triangular bipyramid for $N=5$, and of the regular
octahedron for $N=6$.
\end{remark}

\begin{remark}
The volume maximization problem considered here belongs to a broader family of
extremal problems for finite point configurations on the sphere. In
\cite{Hoehner-Ledford-2022}, the authors introduced weighted cone-volume
functionals, which generalize the classical volume and surface area functionals
of polytopes. An important special case is given by the $L_p$ surface area:
for $0\leq p\leq 1$, this family interpolates between volume and surface area. Sharp inequalities and equality conditions were established for several
classes of inscribed polytopes; in particular, the regular simplex was shown
to maximize the $L_p$ surface area among all simplices inscribed in the sphere
for every $p\in[0,1]$.    
\end{remark}


\subsection{Applications}

Although the problem is classical in convex and discrete geometry, maximum volume spherical configurations also arise in applications. In crystallography and crystal
chemistry, maximum volume polyhedra inscribed in a coordination sphere are used as ideal reference configurations for quantifying the distortion of observed coordination polyhedra \cite{Makovicky}. Of particular relevance to the present work, the tricapped trigonal prism is a standard ninefold coordination polyhedron. In fact, its numerically determined maximum volume realization has been used as an ideal reference in structural studies. For example, in a 2011 X-ray diffraction study of natrite \cite{Ballirano-2011}, Ballirano analyzes an $\text{NaO}_9$  coordination polyhedron by comparing its sphere to polyhedron volume ratio 2.26 with the value 2.05 for the tricapped trigonal prism, explicitly described on \cite[p. 371]{Ballirano-2011} as ``the maximum-volume 9-coordinated polyhedron''. To the best of our knowledge, an analytic proof of this fact has not yet been put forth until now. We address this gap in Theorem~\ref{mainThm}, confirming with an analytic proof that this configuration is in fact the global maximizer among all polyhedra with nine vertices inscribed in the sphere. 

Likewise, in the crystal-chemical study \cite{Friis-BalicZunic-Pekov-Petersen-2004} of kuannersuite-(Ce), the authors compare ninefold coordination environments to the maximum volume tricapped trigonal prism, using the resulting volume distortion to assess the regularity of rare-earth and sodium coordination sites.

Maximum volume spherical point configurations have also appeared in wireless communications. In the design of limited-feedback beamforming codebooks for two-transmit-antenna MIMO systems, the relevant Grassmannian quantization problem can be identified with a point configuration problem on $\mathbb{S}^2$, and maximal volume spherical
codes have been used as precoding codebooks; see
\cite{Pitaval-Maattanen-Schober-Tirkkonen-Wichman}.

     \section{Background and notation}\label{background}

 The standard inner product of two vectors $x=(x_1,x_2,x_3),y=(y_1,y_2,y_3)\in\R^3$ is denoted $\langle x, y\rangle=x_1 y_1+x_2y_2+x_3y_3$, and the Euclidean norm of $x$ is $\|x\|=\sqrt{\langle x,x \rangle}=\sqrt{x_1^2+x_2^2+x_3^2}$. The unit sphere $\Sp$ in $\R^3$ is given by $\Sp=\{(x_1,x_2,x_3)\in\R^3:\,x_1^2+x_2^2+x_3^2=1\}$. 

The convex hull of a set $A\subset\R^3$ is denoted $\conv(A)$. When $A=\{x_1,\ldots,x_N\}\subset\R^3$ is a finite point set, $P=\conv(A)$ is called a \emph{polytope}, and we also write $\conv(A)=[x_1,\ldots,x_N]$. Let $\mathcal{P}_N$ be the set of all polytopes in $\R^3$ with at most $N$ vertices, each lying on the unit sphere $\Sp$. 

A \emph{face} of a polytope $P$ is the intersection of $P$ with a support hyperplane $H$ of $P$ (i.e., $H$ has codimension 1, $H\cap P\neq\varnothing$, and $P$ is contained in a halfspace of $H$). The faces of $P$ of dimension 0, 1, and $2$ are called vertices, edges, and facets, respectively.  For a vertex $v$ of a polytope $P\in\mathcal{P}_N$, the \emph{star of $v$} is the union of all facets of $P$ that contain $v$. For a polytope $P\subset\R^3$, we let $\vol(P)$ denote the $3$-dimensional volume of $P$. The boundary of $P$ is denoted $\partial P$. 


Two polytopes $P,Q\subset\R^3$ are \emph{combinatorially equivalent} if there exists a bijection $\varphi$ between the set of faces $\mathcal{F}(P)$ of $P$ and the set of faces $\mathcal{F}(Q)$ of $Q$ that preserves inclusions, meaning for any two faces $F_1,F_2\in\mathcal{F}(P)$ with $F_1\subset F_2$, we have $\varphi(F_1)\subset\varphi(F_2)$. A polytope $P\subset\R^3$ is \emph{simplicial} if all of its facets are triangles. We let $\mathcal{F}_2(P)$ denote the set of facets of $P$. An ordering of the vertices $(a,b,c)$ of a triangular facet $[a,b,c]\in\mathcal{F}_2(P)$ is \emph{outward oriented} if the normal vector $(b-a)\times(c-a)$ points out of $P$.    For more background on convex polytopes and convex geometry, we refer the reader to, e.g., the books \cite{Brondsted,Grunbaum,SchneiderBook}, and for discrete geometry, see, e.g.,  \cite{VDG,Toth,FejesToth-FejesToth-Kuperberg-2023}. 


\section{Wimmer's reduction and the oriented volume formula}\label{sec:lemmas}

\subsection{Wimmer's combinatorial reduction}

The volume of the convex hull is continuous on the compact set $(\Sp)^9$, so a volume maximizer in $\mathcal{P}_9$ exists. Its volume is positive, since, for example, an inscribed tetrahedron has positive volume. Such a maximizer has exactly nine distinct vertices: if it had fewer, then another point of $\Sp$ outside its convex hull could be added, strictly increasing its volume.

The following reduction to a single combinatorial type is due to Wimmer \cite[Chapter I, Section 7; Chapter II, Section 1, pp.~16--17]{Wimmer1998}.

\begin{lemma}\label{lem:Wimmer-N9-reduction}
    Let $P\in\mathcal{P}_9$ have maximum volume. Then $P$ is combinatorially equivalent to the triaugmented triangular prism.
\end{lemma}

Wimmer obtains this conclusion using vertex-star volume bounds and exclusions of local configurations. In particular, his estimates exclude trivalent vertices and eliminate the other combinatorial types. We use his reduction as stated, and devote the rest of the paper to optimizing the remaining type. His subsequent calculation under symmetry assumptions \cite[Chapter III, Section 4, pp.~33--34]{Wimmer1998} identifies the equality configuration which our unrestricted estimate will recover.

\subsection{Volume decomposition of oriented simplicial polytopes}

The next ingredient we will need is a volume decomposition formula for an oriented simplicial polytope.

\begin{lemma}\label{lem:simp-vol}
    Let $P\subset\R^3$ be a simplicial polytope. If every facet $F=[a,b,c]$ of $P$ is given by its outward orientation, then
    \begin{equation}\label{eq:simp-vol}
        6\vol(P)=\sum_{[a,b,c]\in\mathcal{F}_2(P)}\det(a,b,c).
    \end{equation}
\end{lemma}

\begin{remark}
    Note that this lemma holds even if $o\not\in P$.
\end{remark}

We include a proof of the lemma for completeness.
\begin{proof}[Proof of Lemma \ref{lem:simp-vol}]
    Consider the vector field $X(z)=z/3$. Since $\operatorname{div}(X)=1$, by the divergence theorem we have
    \[
\vol(P)=\frac{1}{3}\int_{\partial P}\langle z,n(z)\rangle\,dA(z).
    \]
    Fix an outward oriented facet $F=[a,b,c]\in\mathcal{F}_2(P)$, and let $n_F$ denote its outer unit normal. Since $F$ lies in a plane, we have that for all $z\in F$, $\langle z,n_F\rangle=h_F$ for some constant $h_F$. Hence $\frac{1}{3}\int_F\langle z,n_F\rangle \,dA(z)=\frac{1}{3}h_F\area(F)$. Since $(a,b,c)$ is the outward boundary orientation, we have $(b-a)\times(c-a)=2\area(F)n_F$. Taking the inner product with $a$, we get $2\area(F)h_F=\langle a,b\times c\rangle=\det(a,b,c)$. Therefore,
    \[
\frac{1}{3}\int_F\langle z,n_F\rangle\,dA(z)=\frac{1}{6}\det(a,b,c).
    \]
    Finally, summing over all facets of $P$, we obtain
    \[
\vol(P)=\sum_{F\in\mathcal{F}_2(P)}\frac{1}{3}\int_F\langle z,n_F\rangle\,dA(z)=\frac{1}{6}\sum_{[a,b,c]\in\mathcal{F}_2(P)}\det(a,b,c).
    \]
\end{proof}

\section{Proof of Theorem \ref{mainThm}}

Let $\mathcal{T}$ denote the set of all polytopes in $\mathcal{P}_9$ which are combinatorially equivalent to the triaugmented triangular prism. As mentioned in the introduction, Wimmer \cite{Wimmer1998} eliminated all other combinatorial types, showing that the global volume maximizer in $\mathcal{P}_9$ must be a member of $\mathcal{T}$. Thus, to complete the proof of Theorem \ref{mainThm}, it remains to determine the volume maximizer in $\mathcal{T}$ under no symmetry assumptions. We will prove this in the following theorem.

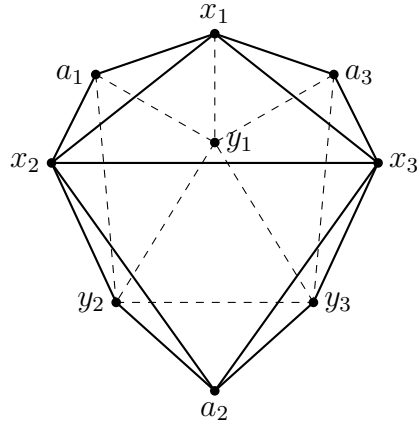
\begin{figure}[h]
\begin{center}
\begin{tikzpicture}[scale=0.9]

\coordinate (x1) at (0,2.8);
\coordinate (x2) at (-2.4,0.9);
\coordinate (x3) at (2.4,0.9);

\coordinate (y1) at (0,1.2);
\coordinate (y2) at (-1.45,-1.15);
\coordinate (y3) at (1.45,-1.15);

\coordinate (a1) at (-1.75,2.2);
\coordinate (a2) at (0,-2.45);
\coordinate (a3) at (1.75,2.2);

\begin{scope}[dashed]
    \draw (y1)--(y2);
    \draw (y2)--(y3);
    \draw (y3)--(y1);

    \draw (x1)--(y1);

    \draw (a1)--(y1);
    \draw (a1)--(y2);

    \draw (a3)--(y1);
    \draw (a3)--(y3);
\end{scope}

\begin{scope}[thick]
    \draw (x1)--(x2);
    \draw (x2)--(x3);
    \draw (x3)--(x1);

    \draw (a1)--(x1);
    \draw (a1)--(x2);

    \draw (a2)--(x2);
    \draw (a2)--(x3);
    \draw (a2)--(y2);
    \draw (a2)--(y3);
        \draw (x2)--(y2);

    \draw (a3)--(x1);
    \draw (a3)--(x3);
        \draw (x3)--(y3);

\end{scope}

\fill (x1) circle (2pt);
\fill (x2) circle (2pt);
\fill (x3) circle (2pt);
\fill (y1) circle (2pt);
\fill (y2) circle (2pt);
\fill (y3) circle (2pt);
\fill (a1) circle (2pt);
\fill (a2) circle (2pt);
\fill (a3) circle (2pt);

\node[above] at (x1) {$x_1$};
\node[left] at (x2) {$x_2$};
\node[right] at (x3) {$x_3$};

\node[right] at (y1) {$y_1$};
\node[left] at (y2) {$y_2$};
\node[right] at (y3) {$y_3$};

\node[left] at (a1) {$a_1$};
\node[below] at (a2) {$a_2$};
\node[right] at (a3) {$a_3$};

\end{tikzpicture}
\end{center}
  \caption{A triaugmented triangular prism.}
    \label{fig:TATP}
\end{figure}

\begin{theorem}\label{thm:TATP}
    Let $P\in\mathcal{T}$. Then $\vol(P)\leq 3\sqrt{2\sqrt{3}-3}\approx 2.04375$ with equality if and only if $P$ is a rotation of the triaugmented triangular prism with vertices  
 \begin{align*}
     a_1&=(1,0,0), & a_2&=\left(-\tfrac{1}{2},\tfrac{\sqrt{3}}{2},0\right), & a_3&=\left(-\tfrac{1}{2},-\tfrac{\sqrt{3}}{2},0\right),\\
     x_1&=\left(\tfrac{1}{2}r_*,\tfrac{\sqrt{3}}{2}r_*,h_*\right), & x_2&=(-r_*,0,h_*), & x_3&=\left(\tfrac{1}{2}r_*,-\tfrac{\sqrt{3}}{2}r_*,h_*\right),\\
     y_1&=\left(\tfrac{1}{2}r_*,\tfrac{\sqrt{3}}{2}r_*,-h_*\right), & y_2&=(-r_*,0,-h_*), & y_3&=\left(\tfrac{1}{2}r_*,-\tfrac{\sqrt{3}}{2}r_*,-h_*\right),
 \end{align*}
 where
 \begin{equation}
     h_*:=\sqrt{2\sqrt{3}-3}\approx 0.681\qquad\text{and}\qquad r_*:=\sqrt{1-h_*^2}=\sqrt{4-2\sqrt{3}} =\sqrt{3}-1\approx 0.732.
 \end{equation}
\end{theorem}
The rest of this section is spent proving this theorem.

\vspace{2mm}

Note that if $P\in\mathcal{T}$, then $P$ has 9 vertices and 14 facets. Let the six degree-5 vertices of $P$ be $x_1,x_2,x_3$ and $y_1,y_2,y_3$, so that $[x_1,x_2,x_3]$ and $[y_1,y_2,y_3]$ are facets of $P$. Denote the three degree-4 vertices of $P$ by $a_1,a_2,a_3$. We choose the labeling so that $(x_1,x_2,x_3)$ is the positive, outward boundary orientation of the triangular facet $[x_1,x_2,x_3]$. In what follows, indices are taken modulo 3. Note that the $i$th degree-4 vertex $a_i$ augments the quadrilateral cycle $(x_i,x_{i+1},y_{i+1},y_i)$. 

\begin{lemma}\label{lem:TATP-vol-vars}
    Let $P\in\mathcal{T}$. Then
    \begin{equation}\label{eq:TATP-vol-vars}
        6\vol(P)=\det(x_1,x_2,x_3)-\det(y_1,y_2,y_3)-\sum_{i=1}^3\langle a_i,W_i\rangle
    \end{equation}
    where $W_i:=(x_i-y_{i+1})\times(x_{i+1}-y_i)$.
\end{lemma}

\begin{proof}[Proof of Lemma \ref{lem:TATP-vol-vars}]
    Let $P\in\mathcal{T}$. Using the notation above, $P$ has 14 triangular facets (see Figure~\ref{fig:TATP}):
    \begin{itemize}
        \item the ``upper'' and ``lower'' triangular facets $[x_1,x_2,x_3]$ and $[y_1,y_2,y_3]$, respectively; 
        \item for $i=1,2,3$, the four triangular facets incident with $a_i$. 
    \end{itemize}

    \noindent We already chose $(x_1,x_2,x_3)$ as the outward orientation of the first triangular facet $[x_1,x_2,x_3]$. This facet's contribution to $\vol(P)$ is $\det(x_1,x_2,x_3)$. Since the two triangular facets of the underlying ``core triangular prism'' of $P$ have opposite boundary orientations, the outward orientation of $[y_1,y_2,y_3]$ is $(y_1,y_3,y_2)$. Thus, its contribution to $\vol(P)$ is $\frac{1}{6}\det(y_1,y_3,y_2)=-\frac{1}{6}\det(y_1,y_2,y_3)$.

    Now fix an index $i$. The four facets incident with $a_i$ are $[a_i,x_i,y_i]$, $[a_i,y_i,y_{i+1}]$, $[a_i,y_{i+1},x_{i+1}]$, and $[a_i,x_{i+1},x_i]$. Their outward boundary orientations can be chosen to be  
    \begin{equation}\label{eq:orientation-check}
        (a_i,x_i,y_i),\quad (a_i,y_i,y_{i+1}), \quad  (a_i,y_{i+1},x_{i+1}), \quad \text{and}\, (a_i,x_{i+1},x_i),
    \end{equation} 
    respectively. Indeed, let us briefly justify that these orientations are mutually consistent. For example, the first two triangles share the edge $[a_i,y_i]$. In the oriented triangle $(a_i,x_i,y_i)$, this edge is traversed $y_i\to a_i$, whereas in the oriented triangle $(a_i,y_i,y_{i+1})$ it is traversed $a_i\to y_i$. Thus the induced orientations are opposite, as required. Similarly, all internal edges from $a_i$ receive opposite orientations from the two incident facets. On the top boundary edge $[x_i,x_{i+1}]$, the top facet $(x_1,x_2,x_3)$ traverses $x_i\to x_{i+1}$, while the triangle $(a_i,x_{i+1},x_i)$ traverses $x_{i+1}\to x_i$; again, the orientations are opposite, as required. Likewise, the lateral fan traverses $y_i\to y_{i+1}$ along the bottom edge, while the bottom facet cycle $(y_1,y_3,y_2)$ traverses the same edge in the opposite direction. Therefore, \eqref{eq:orientation-check} gives the correct outward boundary orientation of $P$, once the global orientation has been fixed by $(x_1,x_2,x_3)$.

    The contribution to $\vol(P)$ of the four facets in \eqref{eq:orientation-check} is
    \begin{align*}
        L_i :&=\frac{1}{6}\big[\det(a_i,x_i,y_i)+\det(a_i,y_i,y_{i+1})+\det(a_i,y_{i+1},x_{i+1})+\det(a_i,x_{i+1},x_i)\big]=\frac{1}{6}\langle a_i,Z_i\rangle
    \end{align*}
    where $Z_i:=x_i\times y_i+y_i\times y_{i+1}+y_{i+1}\times x_{i+1}+x_{i+1}\times x_i$. Expanding bilinearly and using the antisymmetry of the cross product, we obtain
\begin{align*}
    W_i=(x_i-y_{i+1})\times(x_{i+1}-y_i)&=x_i\times x_{i+1}-x_i\times y_i-y_{i+1}\times x_{i+1}+y_{i+1}\times y_i\\
    &=-x_{i+1}\times x_i-x_i\times y_i-y_{i+1}\times x_{i+1}-y_i\times y_{i+1}=-Z_i.
\end{align*}
Hence $L_i=-\langle a_i,W_i\rangle$. Therefore, the total contribution to $\vol(P)$ from the 12 lateral facets is $L_1+L_2+L_3=-\sum_{i=1}^3\langle a_i,W_i\rangle$. Finally, by Lemma \ref{lem:simp-vol} we obtain
\[
6\vol(P)=\det(x_1,x_2,x_3)-\det(y_1,y_2,y_3)-\sum_{i=1}^3\langle a_i,W_i\rangle.
\]
\end{proof}

Next, note that by the Cauchy--Schwarz inequality,
\[
-\langle a_i,W_i\rangle \leq |\langle a_i,W_i\rangle|\leq\|a_i\|\cdot\|W_i\|=\|W_i\|
\]
since $a_i\in\Sp$. Thus,
\begin{equation}\label{eq:TATP-6-vars}
    6\vol(P)\leq \det(x_1,x_2,x_3)-\det(y_1,y_2,y_3)+\sum_{i=1}^3\|(x_i-y_{i+1})\times(x_{i+1}-y_i)\|.
\end{equation}
Equality in \eqref{eq:TATP-6-vars} holds if and only if $\langle a_i,W_i\rangle=-\|W_i\|$ for every $i$. Equivalently, whenever $W_i\neq o$, equality requires
$a_i=-W_i/\|W_i\|$. 

The main lemma of this section is the following
\begin{lemma}\label{lem:TATP-main-est}
    The right-hand side of \eqref{eq:TATP-6-vars} is at most $18\sqrt{2\sqrt{3}-3}$. Equality holds if and only if, after a rotation, there exist unit vectors
$n\in\Sp$ and $u_1,u_2,u_3\in n^\perp\cap\Sp$
such that $\langle u_i,u_j\rangle=-1/2$ for $i\neq j$, $\det(u_1,u_2,n)>0$, 
and
\[
x_i=r_*u_i+h_*n,\qquad y_i=r_*u_i-h_*n,
\qquad i\in\{1,2,3\}.
\]
\end{lemma}

\begin{proof}
\noindent{\bf Step 1: Set up.}    Set $m_i:=\frac{x_i+y_i}{2}$ and $d_i:=\frac{x_i-y_i}{2}$. Then $x_i=m_i+d_i$ and $y_i=m_i-d_i$. Since $x_i,y_i\in\Sp$, we have
    \begin{equation}\label{eq:mi-di}
        \langle m_i,d_i\rangle=0\qquad\text{and}\qquad \|m_i\|^2+\|d_i\|^2=1.
    \end{equation}
    We define the following means and centered vectors:
    \begin{align*}
        \overline{m}:=\frac{m_1+m_2+m_3}{3},\quad \overline{d}:=\frac{d_1+d_2+d_3}{3},\quad \widehat{m}_i:=m_i-\overline{m},\quad \widehat{d}_i:=d_i-\overline{d}.
    \end{align*}
Note that $\sum_{i=1}^3\widehat{m}_i=\sum_{i=1}^3\overline{d}_i=o$. We also set:
\[
U:=\sqrt{3}\|\overline{m}\|,\quad R:=\sqrt{\sum_{i=1}^3\|\widehat{m}_i\|^2},\quad V:=\sqrt{3}\|\overline{d}\|,\quad S:=\sqrt{\sum_{i=1}^3\|\widehat{d}_i\|^2}.
\]
Since $m_i=\overline{m}+\widehat{m}_i$ and $\sum_{i=1}^3\widehat{m}_i=o$, we have
\begin{align*}
    \sum_{i=1}^3\|m_i\|^2 &=\sum_{i=1}^3\langle m_i,m_i\rangle=\sum_{i=1}^3\langle\overline{m}+\widehat{m}_i,\overline{m}+\widehat{m}_i\rangle=\sum_{i=1}^3\left(\|\overline{m}\|^2+2\langle\overline{m},\widehat{m}_i\rangle+\|\widehat{m}_i\|^2\right)\\
    &=3\|\overline{m}\|^2+2\left\langle\overline{m},\sum_{i=1}^3\widehat{m}_i\right\rangle+\sum_{i=1}^3\|\widehat{m}_i\|^2=U^2+R^2.
\end{align*}
Similarly, $\sum_{i=1}^3\|d_i\|^2=V^2+S^2$. Thus, by \eqref{eq:mi-di} we get 
    \begin{equation}\label{eq:URVS-equal-3}
U^2+R^2+V^2+S^2=\sum_{i=1}^3\left(\|m_i\|^2+\|d_i\|^2\right)=3.
    \end{equation}

    \vspace{2mm}

    \noindent{\bf Step 2: Estimating the determinant difference.} Now set
    \[
D:=\det(x_1,x_2,x_3)-\det(y_1,y_2,y_3).
    \]
For triples $Z=(z_1,z_2,z_3)\in(\R^3)^3$ and $W=(w_1,w_2,w_3)\in(\R^3)^3$ satisfying $\sum_{i=1}^3 z_i=\sum_{i=1}^3 w_i=o$, define the vectors
\[
\mathcal{A}(Z):=\sum_{i=1}^3 (z_i\times z_{i+1})\quad\text{and}\quad \mathcal{C}(Z,W):=\sum_{i=1}^3[z_i\times(w_{i+1}-w_{i-1})],
\]
where the indices are meant modulo 3. Since $\sum_{i=1}^3 w_i=o$, explanding the inner product we have
\begin{align*}
    \sum_{i=1}^3\|w_{i+1}-w_{i-1}\|^2 &=\sum_{i=1}^3\langle w_{i+1}-w_{i-1},w_{i+1}-w_{i-1}\rangle =\sum_{i=1}^3\left(\|w_{i+1}\|^2-2\langle w_{i+1},w_{i-1}\rangle+\|w_{i-1}\|^2\right)\\
    &=\sum_{i=1}^3\left[\left(\|w_{i+1}\|^2+\|w_{i-1}\|^2\right)-2\left\langle\sum_{i=1}^3 w_{i+1},w_{i-1}\right\rangle\right]
    =3\sum_{i=1}^3\|w_i\|^2.
\end{align*}
Thus, by the triangle inequality, the estimate $\|a\times b\|\leq\|a\|\,\|b\|$, and the Cauchy--Schwarz inequality, we obtain
\begin{equation}
    \begin{aligned}
        \|\mathcal{C}(Z,W)\| &=\left\|\sum_{i=1}^3[z_i\times(w_{i+1}-w_{i-1})]\right\|\leq\sum_{i=1}^3\|z_i\times(w_{i+1}-w_{i-1})\|\\
        &\leq\sum_{i=1}^3(\|z_i\|\,\|w_{i+1}-w_{i-1}\|)\leq \sqrt{\sum_{i=1}^3\|z_i\|^2}\sqrt{\sum_{i=1}^3\|w_{i+1}-w_{i-1}\|^2}\\
        &=\sqrt{3}\sqrt{\sum_{i=1}^3\|z_i\|^2}\sqrt{\sum_{i=1}^3\|w_i\|^2}.
    \end{aligned}
\end{equation}

Note also that
\begin{equation}\label{eq:CZZ-ineq}
    \begin{aligned}
        \mathcal{C}(Z,Z)&=\sum_{i=1}^3[z_i\times(z_{i+1}-z_{i-1})]=\sum_{i=1}^3(z_i\times z_{i+1}-z_i\times z_{i-1})
        =2\sum_{i=1}^3(z_i\times z_{i+1})=2\mathcal{A}(Z).
    \end{aligned}
\end{equation}
Hence,
\begin{equation}\label{eq:AZ-ineq}
    \|\mathcal{A}(Z)\|=\frac{1}{2}\|\mathcal{C}(Z,Z)\|\leq\frac{\sqrt{3}}{2}\left(\sqrt{\sum_{i=1}^3\|z_i\|^2}\right)^2=\frac{\sqrt{3}}{2}\sum_{i=1}^3\|z_i\|^2.
\end{equation}
By the multilinearity of the determinant, for every $t\in\R^3$ we have
\begin{equation}\label{eq:det-identity}
\begin{aligned}
    \det(t+z_1,t+z_2,t+z_3)&=\det(t,t,t)+\det(t,t,z_3)+\det(t,z_2,t)+\det(z_1,t,t)\\
    &+\det(t,z_2,z_3)+\det(z_1,t,z_3)+\det(z_1,z_2,t)+\det(z_1,z_2,z_3)\\
    &=\det(t,z_2,z_3)+\det(z_1,t,z_3)+\det(z_1,z_2,t)=\langle t,\mathcal{A}(Z)\rangle.
\end{aligned}
\end{equation}
Here the second equality follows because every determinant with at least two $t$-columns is zero, and because $\det(z_1,z_2,z_3)=0$, which follows since $z_3=-z_1-z_2$. 

Now let $\widehat{M}:=(\widehat{m}_1,\widehat{m}_2,\widehat{m}_3)\in(\R^3)^3$ and $\widehat{D}:=(\widehat{d}_1,\widehat{d}_2,\widehat{d}_3)\in(\R^3)^3$. Applying \eqref{eq:det-identity} twice, once with $t^+=\overline{m}+\overline{d}$ and $z_i^+=\widehat{m}_i+\widehat{d}_i$, and once with $t^-=\overline{m}-\overline{d}$ and $z_i^-=\widehat{m}_i-\widehat{d}_i$, we derive that
\begin{align*}
    D&=\det(m_1+d_1,m_2+d_2,d_3+m_3)-\det(m_1-d_1,m_2-d_2,m_3-d_3)\\
    &=\det(t^+ +z_1^+,t^+ +z_2^+,t^+ +z_3^+)-\det(t^- -z_1^-,t^- - z_2^-, t^- - z_3^-)\\
    &=\langle t^+,\mathcal{A}(Z^+)\rangle-\langle t^-,\mathcal{A}(Z^-)\rangle
\end{align*}
where $Z^\pm=(z_1^\pm,z_2^\pm,z_3^\pm)$. Since $Z^+=\widehat{M}+\widehat{D}$ and $Z^-=\widehat{M}-\widehat{D}$, this implies
\begin{align*}
    D&=\langle t^+,\mathcal{A}(\widehat{M}+\widehat{D})\rangle-\langle t^-,\mathcal{A}(\widehat{M}-\widehat{D})\rangle\\
    &=\langle\overline{m}+\overline{d},\mathcal{A}(\widehat{M})+\mathcal{A}(\widehat{D})+\mathcal{C}(\widehat{M},\widehat{D})\rangle-\langle\overline{m}-\overline{d},\mathcal{A}(\widehat{M})+\mathcal{A}(\widehat{D})-\mathcal{C}(\widehat{M},\mathcal{D})\rangle
\end{align*}
where we used the identities $\mathcal{A}(\widehat{M}\pm\widehat{D})=\mathcal{A}(\widehat{M})+\mathcal{A}(\widehat{D})\pm\mathcal{C}(\widehat{M},\widehat{D})$. Expanding the inner products and using the inequalities $\|a\times b\|\leq\|a\|\,\|b\|$, \eqref{eq:AZ-ineq}, \eqref{eq:CZZ-ineq}, and the triangle inequality, we derive that
    \begin{equation}\label{eq:det-diff-est-final-URVS}
    \begin{aligned}
        D&=2\langle\overline{m},\mathcal{C}(\widehat{M},\widehat{D})\rangle+2\langle\overline{d},\mathcal{A}(\widehat{M})+\mathcal{A}(\widehat{D})\rangle \\
        &\leq 2\|\overline{m}\|\,\|\mathcal{C}(\widehat{M},\widehat{D})\|+2\|\overline{d}\|\,\|\mathcal{A}(\widehat{M})+\mathcal{A}(\widehat{D})\|\\
        &\leq 2\|\overline{m}\|\,\|\mathcal{C}(\widehat{M},\widehat{D})\|+2\|\overline{d}\|(\|\mathcal{A}(\widehat{M})\|+\|\mathcal{A}(\widehat{D})\|)\\
        &\leq 2\left(\frac{VR^2}{2}+\frac{VS^2}{2}+URS\right)=V(R^2+S^2)+2URS.
    \end{aligned}
    \end{equation}
    
    \vspace{2mm}

    \noindent{\bf Step 3: Estimating the cross product terms.} By definition, $x_i-y_{i+1}=(m_i-m_{i+1})+(d_i+d_{i+1})$ and $x_{i+1}-y_i=-(m_i-m_{i+1})+(d_i+d_{i+1})$. Hence
    \begin{align*}
        W_i &=\big[(m_i-m_{i+1})+(d_i+d_{i+1})\big]\times\big[-(m_i-m_{i+1})+(d_i+d_{i+1})\big]\\
        &=2(m_i-m_{i+1})\times(d_i+d_{i+1}).
    \end{align*}
   Again using $\|a\times b\|\leq\|a\|\,\|b\|$ and the Cauchy--Schwarz inequality, we derive that
    \begin{equation}\label{eq:Wi-CS-midi}
        \sum_{i=1}^3\|W_i\| \leq 2\sqrt{\sum_{i=1}^3\|m_i-m_{i+1}\|^2}\sqrt{\sum_{i=1}^3\|d_i+d_{i+1}\|^2}.
    \end{equation}

    Now on one hand,
    \[
\sum_{i=1}^3\|m_i-m_{i+1}\|^2=2\sum_{i=1}^3\|m_i\|^2-2\sum_{i=1}^3\langle m_i,m_{i+1}\rangle.
    \]
    On the other hand, 
    \[
\left\|\sum_{i=1}^3 m_i\right\|^2=\sum_{i=1}^3\|m_i\|^2+2\sum_{i=1}^3\langle m_i,m_{i+1}\rangle. 
    \]
    Thus,
    \begin{equation}\label{eq:Wi-CS-1}
        \sum_{i=1}^3\|m_i-m_{i+1}\|^2=3\sum_{i=1}^3\|m_i\|^2-\left\|\sum_{i=1}^3 m_i\right\|^2=3(U^2+R^2)-9\|\overline{m}\|^2=3R^2.
    \end{equation}
    Similarly, \begin{equation}\label{eq:W-i-CS-2}
        \sum_{i=1}^3\|d_i+d_{i+1}\|^2=\sum_{i=1}^3\|d_i\|^2+\left\|\sum_{i=1}^3 d_i\right\|^2=V^2+S^2+9\|\overline{d}\|^2=4V^2+S^2.
    \end{equation}
    Therefore, by \eqref{eq:Wi-CS-midi}, \eqref{eq:Wi-CS-1} and \eqref{eq:W-i-CS-2}, we obtain
    \begin{equation}\label{eq:Wi-final}
       \sum_{i=1}^3 \|W_i\| \leq 2\sqrt{3R^2}\sqrt{4V^2+S^2}=2\sqrt{3}R\sqrt{4V^2+S^2}.
    \end{equation}
    
    Let  $E(x_1,x_2,x_3,y_1,y_2,y_3)
:=D+\sum_{i=1}^3\|W_i\|$.  
    Putting everything together, by  \eqref{eq:det-diff-est-final-URVS} and \eqref{eq:Wi-final} we obtain
    \begin{equation}\label{eq:FURVS}
E(x_1,x_2,x_3,y_1,y_2,y_3)
\leq
V(R^2+S^2)+2URS+2\sqrt3R\sqrt{4V^2+S^2}
=:F(U,R,V,S).
    \end{equation}
    Consequently, by \eqref{eq:TATP-6-vars},
    \[
6\vol(P)\leq E(x_1,x_2,x_3,y_1,y_2,y_3)
\leq F(U,R,V,S).
\]
    
    \vspace{2mm}

    \noindent {\bf Step 4: Reducing the number of variables.} First, we show that
    \begin{equation}\label{eq:URVS-upper-1}
        VS^2+2URS \leq 2(U^2+S^2).
    \end{equation}
    This is equivalent to $2U^2-2URS+(2-V)S^2\geq 0$. Completing the square, this is equivalent to
    \[
2\left(U-\frac{RS}{2}\right)^2+\left(2-V-\frac{R^2}{2}\right)S^2\geq 0.
    \]
    We claim that $2-V-R^2/2\geq 0$, i.e., $4-2V-R^2\geq 0$. Since $R^2\leq 3-V^2$, we get $4-2V-R^2\geq 4-2V-(3-V^2)=(V-1)^2\geq 0$. This proves \eqref{eq:URVS-upper-1}.

    Next, set $T:=\sqrt{V^2+S^2/4}$. Then $T\geq V$ and $\sqrt{4V^2+S^2}=2T$. Hence, by \eqref{eq:FURVS} and \eqref{eq:URVS-upper-1}, we have $F(U,R,V,S)\leq TR^2+4\sqrt{3}RT+2(U^2+S^2)$. Note also that $3-R^2-T^2=U^2+\frac{3}{4}S^2$. Hence
    \[
2(U^2+S^2)\leq\frac{8}{3}U^2+2S^2=\frac{8}{3}\left(U^2+\frac{3}{4}S^2\right)=\frac{8}{3}(3-R^2-T^2).
    \]
Combining the previous estimates, we derive that $F(U,R,V,S)\leq G(R,T)$, where 
\[
G(R,T):=TR^2+4\sqrt{3}RT+\frac{8}{3}(3-R^2-T^2)
\]
is defined on the closed quarter-disk $\Omega:=\{(R,T): R,T\geq 0, R^2+T^2\leq 3\}$. Indeed, by definition, $R\geq 0$ and $T\geq 0$. Moreover, $T^2=V^2+S^2/4$, so $R^2+T^2=R^2+V^2+S^2/4$. Using $U^2+R^2+V^2+S^2=3$, we get $R^2+V^2=3-U^2-S^2$, which implies $R^2+T^2=3-U^2-3S^2/4$. Since $U^2,S^2\geq 0$, we have $R^2+T^2\leq 3$.

\vspace{2mm}

\noindent {\bf Step 5: Solving the optimization problem.} We have shown that for every admissible $(U,R,V,S)$, we have $F(U,R,V,S)\leq G(R,T)$. Since $(R,T)\in\Omega$, we have $F(U,R,V,S)\leq G(R,T)\leq \max_{(r,t)\in\Omega}G(r,t)$, so 
\[
\max_{\text{admissible }(U,R,V,S)}F(U,R,V,S)\leq \max_{(R,T)\in\Omega}G(R,T).
\]
We aim to solve the optimization problem $\max_{(R,T)\in\Omega}G(R,T)$. Observe that $\Omega$ is compact and $G(R,T)$ is continuous on $\Omega$. Therefore, $G$ attains a maximum on $\Omega$. The partial derivatives of $G$ are
\[
\frac{\partial G}{\partial R}=2RT+4\sqrt{3}T-\frac{16}{3}R\qquad\text{and}\qquad \frac{\partial G}{\partial T}=R^2+4\sqrt{3}R-\frac{16}{3}T.
\]

 
Setting $\frac{\partial G}{\partial T}=0$, we get $T=\frac{3}{16}R^2+\frac{3\sqrt{3}}{4}R$. Plugging this into $\frac{\partial G}{\partial R}$, we get the cubic $\frac{3}{8}R[(R+3\sqrt{3})^2-\frac{155}{9}]$, which has only nonpositive roots: $R=0,-3\sqrt{3}\pm\frac{\sqrt{155}}{3}$.
Therefore, the maximum of $G(R,T)$ is attained on the boundary of $\Omega$.

On the coordinate axes, we have 
\[
G(0,T)=8-\frac{8}{3}T^2\leq 8\qquad \text{and}\qquad G(R,0)=8-\frac{8}{3}R^2\leq 8.
\]    
On the circular boundary $R^2+T^2=3$, set $r:=R/\sqrt{3}$ and $h:=T/\sqrt{3}$. Then $r^2+h^2=1$. Hence $G(R,T)$ reduces to the single-variable function
\[
G(R,T)=3\sqrt{3}(r^2+4r)\sqrt{1-r^2}=:\phi(r).
\]
It remains to maximize $\phi(r)$ on the interval $0\leq r\leq 1$. We have
\[
\phi'(r)=3\sqrt{3}\left(2(r+2)\sqrt{1-r^2}-\frac{r^2(r+4)}{\sqrt{1-r^2}}\right)=0.
\]
This is equivalent to $3r^3+8r^2-2r-4=0$. The roots of this polynomial are $r_1=-2/3$, $r_2=-1+\sqrt{3}$, and $r_3=-1-\sqrt{3}$; only $r_*:=r_2=-1+\sqrt{3}\in[0,1]$, hence it is the only critical point of $\phi$. We have
\[
\phi(-1+\sqrt{3})=18\sqrt{2\sqrt{3}-3}\approx 12.263>8,
\]
and $\phi(0)=\phi(1)=0$. Therefore,
\[
\max_{(R,T)\in\Omega}G(R,T)=\phi(-1+\sqrt{3})=18\sqrt{2\sqrt{3}-3}.
\]
We have thus shown that $6\vol(P)\leq 18\sqrt{2\sqrt{3}-3}$, i.e., $\vol(P)\leq 3\sqrt{2\sqrt{3}-3}$.

\vspace{2mm}

\noindent{\bf Step 6: Equality analysis and identification of combinatorial type.} Assume now that equality holds in Lemma~\ref{lem:TATP-main-est},  that is,
\[
E(x_1,x_2,x_3,y_1,y_2,y_3)
=18\sqrt{2\sqrt{3}-3}.
\]
Then equality must hold throughout the chain
\[
E(x_1,x_2,x_3,y_1,y_2,y_3)\leq F(U,R,V,S)
\leq G(R,T).
\] 
Let $h=h_*:=\sqrt{2\sqrt{3}-3}$ and $r=r_*:=\sqrt{1-h_*^2}$. The unique maximizer of $G$ satisfies $(R,T)=(\sqrt{3}r,\sqrt{3}h)$ and $R^2+T^2=3$. Recalling that $3-R^2-T^2=U^2+\frac{3}{4}S^2$, we deduce that $U^2+\frac{3}{4}S^2=0$, so $U=S=0$ at a maximizer of $G$. The condition $U=0$ says $m_1+m_2+m_3=0$, while $S=0$ says $d_1=d_2=d_3=:d$. Hence by \eqref{eq:mi-di} we have $\langle m_i,d\rangle=0$ and $\|m_i\|^2=1-\|d\|^2$. Thus, the $m_i$ lie in the plane $d^\perp$, have equal (positive) length, and sum to 0. This means that the $m_i$ are the vertices of an equilateral triangle centered at the origin in $d^\perp$. Moreover, $\|d\|=V/\sqrt{3}$ and $3=U^2+R^2+V^2+S^2=R^2+V^2$, so $V=\sqrt{3-R^2}=T$. Hence $\|d\|=T/\sqrt{3}=h$. Also, $\|m_i\|=\sqrt{1-h^2}=r$. Thus, using also that $r^2+h^2=1$ and $x_i,y_i\in\Sp$, we deduce that there exist a unit vector $n\in\Sp$ and a triple $u_1,u_2,u_3\in n^\perp\cap\Sp$ with $\langle u_i,u_j\rangle=-1/2$ for $i\neq j$ such that $x_i=ru_i+hn$ and $y_i=ru_i-hn$. 

Equality in the determinant estimate yields
\[
D=V(R^2+S^2)+2URS=VR^2=TR^2=3\sqrt{3}hr^2.
\]
For these centered equilateral triangles, by a direct expansion we also obtain $D=6hr^2\det(u_1,u_2,n)$. Since $h,r>0$, the cyclic order satisfies $\det(u_1,u_2,n)>0$. For this configuration,
\begin{align*}
    W_i &=\big[(ru_i+hn)-(ru_{i+1}-hn)\big]\times\big[(ru_{i+1}+hn)-(ru_i-hn)\big]\\
    &=\big[r(u_i-u_{i+1})+2hn\big]\times\big[r(u_{i+1}-u_i)+2hn\big]=4rh(u_i-u_{i+1})\times n.
\end{align*}
Since $\|u_i-u_{i+1}\|=\sqrt{3}$, this implies
\[
\|W_i\|=4rh\|(u_i-u_{i+1})\times n\|=4rh\|u_i-u_{i+1}\|\,\|n\|\sin(\measuredangle(u_i-u_{i+1},n))=4\sqrt{3}rh.
\]
Consequently,
\[
D+\sum_{i=1}^3\|W_i\|
=3\sqrt{3}hr^2+12\sqrt{3}rh
=3\sqrt{3}(r^2+4r)\sqrt{1-r^2}.
\]
For $r=r_*=\sqrt{3}-1$ and $h=h_*=\sqrt{1-r_*^2}$, this is equal to $18\sqrt{2\sqrt{3}-3}$.
Thus, equality in the upper bound of Lemma~\ref{lem:TATP-main-est} is attained, and the equality configurations for its right-hand side are precisely
the six-point configurations
\[
x_i=r_*u_i+h_*n,
\qquad
y_i=r_*u_i-h_*n,
\qquad i\in\{1,2,3\},
\]
up to an orthogonal transformation. This completes the proof of Lemma~\ref{lem:TATP-main-est}.
\end{proof}

\begin{proof}[Proof of Theorem~\ref{thm:TATP}]
The upper bound follows immediately from \eqref{eq:TATP-6-vars} and
Lemma~\ref{lem:TATP-main-est}. Suppose now that equality holds in Theorem~\ref{thm:TATP}. Then equality must
hold both in Lemma~\ref{lem:TATP-main-est} and in each of the Cauchy--Schwarz
inequalities used in passing from \eqref{eq:TATP-vol-vars} to \eqref{eq:TATP-6-vars}. Therefore, 
Lemma~\ref{lem:TATP-main-est}  determines the six vertices $x_i,y_i$ as
above. Moreover, equality in \eqref{eq:TATP-6-vars} requires $a_i=-W_i/\|W_i\|$ for $i\in\{1,2,3\}$. Since $W_i=4r_*h_*(u_i-u_{i+1})\times n$, 
we obtain
\[
a_i=\frac{1}{\sqrt{3}}(u_{i+1}-u_i)\times n.
\]
Hence, the apices $a_i$ form the vertices of a regular triangle inscribed in the great circle $n^\perp\cap\Sp$, offset by an angle $\pi/3$ from the $u_i$. For example, if we take $n=e_3$ and $u_1=(\frac{1}{2},\frac{\sqrt{3}}{2},0)$, $u_2=(-1,0,0)$, and $u_3=(\frac{1}{2},-\frac{\sqrt{3}}{2},0)$, after a cyclic relabeling we get $a_1=(1,0,0)$, $a_2=(-\frac{1}{2},\frac{\sqrt{3}}{2},0)$, and $a_3=(-\frac{1}{2},-\frac{\sqrt{3}}{2},0)$, which are the equatorial vertices of the symmetric triaugmented triangular prism. The corresponding $x_i$ and $y_i$ may be computed using the formulas above.

Let $P=\conv\{a_1,a_2,a_3,x_1,x_2,x_3,y_1,y_2,y_3\}$. Let us justify that the polytope that arises from this method has the required combinatorial type, i.e., $P\in\mathcal{T}$. We will directly check that from all supporting planes we get the required facet structure. Throughout, we will use the following identities: $\langle a_i,u_i\rangle=\langle a_i,u_{i+1}\rangle=1/2$, $\langle a_i,u_{i+2}\rangle=-1$, and $\langle a_i,a_j\rangle=-1/2$ if $i\neq j$. 

\vspace{2mm}

\noindent\underline{\it 1. The upper and lower triangular facets.} All three $x_i$ lie in the plane $\langle n,z\rangle=h$. Every other vertex lies strictly below it since  $\langle n,a_i\rangle=0<h$ and $\langle n,y_i\rangle=-h<h$. Hence $X:=\conv\{x_1,x_2,x_3\}$ is a facet of $P$. Similarly, the plane $\langle n,z\rangle=-h$ supports $P$ from below, and $Y:=\conv\{y_1,y_2,y_3\}$ is a facet of $P$. 

\vspace{2mm}

\noindent\underline{\it 2. The upper lateral facets.} Set $t:=\frac{1-r/2}{h}$. Fix $i$ and consider the function $\Phi_i^+(z):=\langle a_i,z\rangle+t\langle n,z\rangle$. Note that: $\Phi_i^+(a_i)=1$, $\Phi_i^+(x_i)=r\langle a_i,u_i\rangle+th=\frac{r}{2}+(1-\frac{r}{2})=1$, and similarly $\Phi_i^+(x_{i+1})=1$. Let us check that every other vertex gives a $\Phi_i^+$-value less than 1. For $j\neq i$, we have $\Phi_i^+(a_j)=\langle a_i,a_j\rangle=-1/2<1$. For the two corresponding lower vertices, we have 
\[
\Phi_i^+(y_i)=\Phi_i^+(y_{i+1})=\frac{r}{2}-th=\frac{r}{2}-\left(1-\frac{r}{2}\right)=r-1<1.
\]
Moreover,
\[
\Phi_i^+(x_{i+2})=r\langle a_i,u_{i+2}\rangle+th
=-r+1-\frac r2
=1-\frac{3r}{2}<1,
\]
and
\[
\Phi_i^+(y_{i+2})=-r-th=-1-\frac{r}{2}<1.
\]
Therefore, $\Phi_i^+(z)\leq 1$ for all $z\in P$, with equality precisely on the set $\{a_i,x_i,x_{i+1}\}$. These three points are noncollinear because $x_i$ and $x_{i+1}$ lie in the plane $\{z=h\}$, whereas $a_i\in\{z=0\}$. Therefore, $A_i:=\conv\{a_i,x_i,x_{i+1}\}$ is a triangular facet of $P$. There are three of these upper lateral facets. 

\vspace{2mm}

\noindent\underline{\it 3. The lower lateral facets.} Define the function $\Phi_i^-(z):=\langle a_i,z\rangle-t\langle n,z\rangle$. We have $\Phi_i^-(a_i)=1$ and
\[
\Phi_i^-(y_i)=r\langle a_i,u_i\rangle+th
=\frac{r}{2}+\left(1-\frac{r}{2}\right)
=1.
\]
Similarly, we also have $\Phi_i^-(y_{i+1})=1$.

We now check that every other vertex has a $\Phi_i^-$-value that is
strictly less than $1$. For $j\neq i$, we have
\[
\Phi_i^-(a_j)=\langle a_i,a_j\rangle=-\frac{1}{2}<1.
\]
For the two corresponding upper vertices,
\[
\Phi_i^-(x_i)=\Phi_i^-(x_{i+1})
=\frac{r}{2}-th
=\frac{r}{2}-\left(1-\frac{r}{2}\right)
=r-1<1.
\]
Moreover, since $\langle a_i,u_{i+2}\rangle=-1$, we also have
\[
\Phi_i^-(x_{i+2})=-r-th
=-r-\left(1-\frac{r}{2}\right)
=-1-\frac{r}{2}<1,
\]
and
\[
\Phi_i^-(y_{i+2})=-r+th
=-r+\left(1-\frac{r}{2}\right)
=1-\frac{3r}{2}<1.
\]
Consequently, $\Phi_i^-(z)\leq 1$ for every $z\in P$, with equality precisely at
$\{a_i,\quad y_i,\quad y_{i+1}\}$. These three points are noncollinear, since $y_i$ and $y_{i+1}$
lie in the plane $\{z:\langle n,z\rangle=-h\}$, whereas $\langle n,a_i\rangle=0$. 
Therefore, $B_i:=\conv\{a_i,y_i,y_{i+1}\}$ 
is a triangular facet of $P$ for each $i$.

\vspace{2mm}

\noindent\underline{\it 4. The first family of vertical lateral facets.} Next, we prove that $C_i:=\conv\{a_i,x_i,y_i\}$ is a facet of $P$. The vertices $x_i$ and $y_i$ have the same orthogonal projection $ru_i$ onto $n^\perp$. Hence the supporting plane of $P$ is vertical. Define the horizontal vector $c_i:=(2r-1)a_i+(2-r)u_i$, and consider the function $\Psi_i(z):=\langle c_i,z\rangle$. Since $c_i\perp n$, the function $\Psi_i$ has the same value at $x_j$ and $y_j$. Moreover, 
\[
\Psi_i(a_i)=(2r-1)\langle a_i,a_i\rangle+(2-r)\langle u_i,a_i\rangle=(2r-1)+\frac{2-r}{2}=\frac{3r}{2}, 
\]
and
\[
\Psi_i(x_i)=\Psi_i(y_i)=r\langle c_i,u_i\rangle=r\left(\frac{2r-1}{2}+2-r\right)=\frac{3r}{2}.
\]
We evaluate $\Psi_i$ at the other projected vertices: 
\begin{itemize}
    \item Since $\langle a_i,a_{i-1}\rangle=-1/2$ and $\langle u_i,a_{i-1}\rangle=1/2$, we have $\Psi_i(a_{i-1})=-\frac{2r-1}{2}+\frac{2-r}{2}=\frac{3}{2}(1-r)$. Since $r>1/2$, we have $\frac{3}{2}(1-r)<\frac{3r}{2}$.

    \item Since $\langle a_i,a_{i+1}\rangle=-1/2$ and $\langle u_i,a_{i+1}\rangle=-1$, we have \[\Psi_i(a_{i+1})=-\frac{2r-1}{2}-(2-r)=-\frac{3}{2}<\frac{3r}{2}.\]

    \item Since $\langle a_i,u_{i+1}\rangle=1/2$ and $\langle u_i,u_{i+1}\rangle=-1/2$, we have \[\Psi_i(x_{i+1})=\Psi_i(y_{i+1})=r\left(\frac{2r-1}{2}-\frac{2-r}{2}\right)=\frac{3r(r-1)}{2}<\frac{3r}{2}.\]

    \item Since $\langle a_i,u_{i+2}\rangle=-1$ and $\langle u_i,u_{i+2}\rangle=-1/2$, we have \[\Psi_i(x_{i+2})=\Psi_i(y_{i+2})=r\left(-(2r-1)-\frac{2-r}{2}\right)=-\frac{3r^2}{2}<\frac{3r}{2}.\]
\end{itemize}
Therefore, $\Psi_i(z)\leq 3r/2$ for all $z\in P$, with equality occurring precisely on the set $\{a_i,x_i,y_i\}$. These three points are noncollinear since $[x_i,y_i]$ is a vertical segment, while $a_i$ is not on its support line because $a_i\neq ru_i$. Therefore, $C_i$ is a triangular facet of $P$. 

\vspace{2mm}

\noindent\underline{\it 5. The second family of vertical lateral facets.} Define $\widetilde{c}_i:=(2r-1)a_i+(2-r)u_{i+1}$ and $\widetilde{\Psi}_i(z):=\langle\widetilde{c}_i,z\rangle$. Interchanging the roles of $u_i$ and $u_{i+1}$ in the preceding case, we get \[\widetilde{\Psi}_i(a_i)=\widetilde{\Psi}_i(x_{i+1})=\widetilde{\Psi}_i(y_{i+1})=\frac{3r}{2},\]
and as before, every other vertex gives a $\widetilde{\Psi}_i$-value less than $3r/2$. It follows that $D_i:=\conv\{a_i,x_{i+1},y_{i+1}\}$ is a facet of $P$.

\vspace{2mm}

\noindent\underline{\it 6. There are no additional facets.} We have exhibited $2+3(4)=14$ distinct facets of $P$:
\[
X, Y, A_1, A_2, A_3, B_1, B_2, B_3, C_1, C_2, C_3, D_1, D_2, D_3.
\]
All nine given points are vertices of $P$ as exposed points of their respective facets. Thus $f_0(P)=9$. Any convex 3-polytope with $f_0(P)=9$ must have $f_2(P)\leq 14$. Indeed, $3f_2(P)\leq 2f_1(P)$ since every facet contains at least 3 edges and every edge belongs to exactly 2 facets. By Euler's formula, $9-f_1(P)+f_2(P)=2$, so $f_1(P)=f_2(P)+7$. Hence $3f_2(P)\leq 2(f_2(P)+7)$, so $f_2(P)\leq 14$. 

\vspace{2mm}

\noindent\underline{\it 7. Identification of the combinatorial type.} The six vertices $x_1$, $x_2$, $x_3$, $y_1$, $y_2$, and $y_3$ form the vertices of a triangular prism. Its $i$th lateral rectangular facet has cyclic vertex set $(x_i,x_{i+1},y_{i+1},y_i)$. The apex point $a_i$ lies beyond that rectangular face and replaces it by four triangles: $[a_i, x_i,x_{i+1}]$, $[a_i,x_{i+1},y_{i+1}]$, $[a_i,y_{i+1},y_i]$, and $[a_i,y_i,x_i]$. This occurs on all three lateral facets of the prism, hence $P\in\mathcal{T}$. 

\vspace{2mm}

This completes the proof of Theorem \ref{thm:TATP}. 
\end{proof}

\section*{Acknowledgments}

We are grateful to Lienhard Wimmer for reading the first version of this paper and bringing our attention to his Ph.D. thesis, and for suggesting some improvements to our arguments.

\bibliographystyle{plain}
\bibliography{REVISED}


\vspace{3mm}

\noindent {\sc Department of Mathematics \& Computer Science, Longwood University, U.S.A.}

\noindent {\it E-mail addresses:} {\tt hoehnersd@longwood.edu, ledfordjp@longwood.edu}

\end{document}